\documentclass[11pt]{amsart}
\usepackage[left=0.55in,right=0.55in,top=0.59in,bottom=0.55in]{geometry}

\usepackage{amscd}
\usepackage{latexsym}
\usepackage[usenames]{color}
\usepackage{mathrsfs}
\usepackage{amsmath,amsfonts,amssymb,amsthm}
\makeindex

\newtheorem{theorem}{Theorem}[section]
\newtheorem{lemma}[theorem]{Lemma}
\newtheorem{notation}[theorem]{Notation}
\newtheorem{corollary}[theorem]{Corollary}

\newtheorem{proposition}[theorem]{Proposition}
\theoremstyle{definition}

\newtheorem{definition}[theorem]{Definition}
\newtheorem{example}[theorem]{Example}

\newtheorem{question}[theorem]{Question}

\newtheorem{remark}[theorem]{Remark}

\numberwithin{equation}{section}

\newcommand{\T}{\mathbb{T}}
\newcommand{\Z}{\mathbb{Z}}

\newcommand{\N}{\mathbb{N}}

\newcommand{\R}{\mathbb{R}}

\renewcommand{\theta}{\vartheta}
\newcommand{\eps}{\varepsilon}
\renewcommand{\phi}{\varphi}

\newcommand{\Hom}{{\rm Hom}}
\newcommand{\id}{{\rm id}}
\newcommand{\dis}{\displaystyle}

\newcommand{\cal}{\mathcal}

\def\SS_{{\mathfrak S}_{qc}}

\begin{document}
\title{Linear topologies on $\Z$ are not Mackey topologies }

\author{Lydia Aussenhofer}
\address{Universit\"at Passau, Fakult\"at f\"ur Informatik und Mathematik, Innstr. 33, D-94032 Passau}
\email{lydia.aussenhofer@uni-passau.de}

\author{Daniel de la Barrera Mayoral}
\address{Departamento de Geometr\'{\i}a y Topolog\'{\i}a, Universidad Complutense de Madrid,
E-28040 Madrid, Spain}
\email{danielbarreramayoral@gmail.com}

\thanks{The  second named author wishes to thank   the Katholische Universit\"at Eichst\"att/Ingolstadt for
an invitation which enabled him to work on this topic.}




\begin{abstract}
Given a linear non-discrete topology $\lambda$ on the integers, it will be shown, that there exists a strictly finer metrizable locally quasi-convex group topology   $\tau$ on $\Z$ such that $(\Z,\lambda)^\wedge=(\Z,\tau)^\wedge$ (algebraically).

Applying this result to the $p$-adic topology on $\Z$, we give a negative answer to the question of Dikranjan, whether
this topology is Mackey.
\end{abstract}

\maketitle


\section{Introduction and Notation}

  In the framework of the theory of topological vector spaces the ''Mackey topology" is an important topic.
  Given a   $K$-vector space $ E $ and a (point-separating) subspace $F$ of ${\cal L}(E,K)$, the
  algebraic dual of $E$, there exists a finest (Hausdorff) vector space topology $\mu_F$ on $E$ such that
  $\mu_F$ is locally convex and the topological dual of $(E,\mu_F)$ is exactly $F$. The existence is
  a direct consequence of the Hahn-Banach theorem:  it is the supremum of all locally convex vector space topologies
  with topological dual $F$. If $(E,\tau)$ is a metrizable locally convex vector space and $F$ is the topological dual $(E,\tau)'$ then $\tau$ coincides with $\mu_F$.

  In \cite{mackey} the analogue of the Mackey topology for locally quasi-convex groups has been introduced and in \cite{BK} a categorical
  approach is presented.

   Given an abelian topological group $(G,\tau)$, the set of all continuous homomorphisms (or equivalently:
  {\bf continuous characters}) $\chi:G\to \T$, where $\T$ denotes the
  quotient group $\R/\Z$  (which is topologically isomorphic to the compact group of complex numbers of modulus $1$), forms an abelian group when addition is defined pointwise.
  It  is called {\bf dual group} or {\bf character group} of $G$ and denoted by $(G,\tau)^\wedge$ or simply $G^\wedge$ if the topology is clear from the
  context.

  Let  $H$ be a subgroup  of the homomorphism group  ${\rm Hom}(G,\T)$.
  The coarsest group topology on $G$ with dual $H$ is the initial topology on $G$ induced by the mapping $G\to\T^H,\ x\mapsto(\chi(x))_{\chi\in H}$.
  This topology will be denoted by $\sigma(G,H)$.
  It is natural to ask
  whether there exists a finest group topology   on $G$  with character group  $ H$.
  As in the case of topological  vector spaces, this question is usually restricted to the class of locally quasi-convex groups,
  a notion introduced by Vilenkin in \cite{vilenkin}, which generalized the setting of local convexity in topological vector spaces.
   For subsets $A\subseteq G$ and $B\subseteq G^\wedge$  the set $A^\triangleright=\{\chi\in G^\wedge:\ \chi(A)\subseteq \T_+\}$ is called the {\bf polar} of $A$ and
  $B^\triangleleft=\{x\in G:\ \forall\;\chi\in B\ \ \ \chi(x)\in\T_+\}$ is called
 the {\bf prepolar} of $B$, where  $\T_+=[-\frac{1}{4} ,\frac{1}{4}]+\Z $.
A subset $A$ of an abelian topological group is called {\bf   quasi-convex} if $A=(A^\triangleright)^\triangleleft$ holds.
This means that for every $x\notin A$ there exists a continuous character $\chi\in A^\triangleright$ such that $\chi(x)\notin\T_+$;
this generalizes  the description of closed, convex, and  symmetric subsets in vector spaces given by the Hahn-Banach theorem.
 An abelian topological group is called {\bf locally quasi-convex} if it has a neighborhood basis at the neutral element $0$ consisting of quasi-convex sets. Since the class of locally quasi-convex groups is stable under taking arbitrary products and subgroups, the topology $\sigma(G,H)$ is
 locally quasi-convex.

 Given two locally quasi-convex group topologies $\lambda$ and $\tau$ on an abelian group $G$, then $\lambda$ is called {\bf compatible} with $ \tau $ if
 $(G,\tau)^\wedge=(G,\lambda)^\wedge$ holds.  Of course, being compatible defines an equivalence relation on the set of all
  locally quasi-convex group topologies on a given abelian group $G$. If $H=(G,\tau)^\wedge$, then the precompact topology $\sigma(G,H)$ is the coarsest among all group topologies
  compatible with $\tau$.
  In case there exists a finest group topology compatible with $\tau$, this topology is called
  {\bf Mackey topology} for $(G,\tau)$ or for the pair $(G,H)$ and it is denoted by $\mu_\tau$ or $\mu_H$ when $H=(G,\tau)^\wedge$.

In \cite{mackey} it was shown that if $(G,\tau)$ is a complete and metrizable locally quasi-convex group, then
  $\tau$ is the   Mackey topology.  It is not known whether every locally quasi-convex group  admits a Mackey topology.

As mentioned above, for a pair $(G,H)$, the precompact topology $\sigma(G,H)$ is the coarsest locally quasi-convex group topology with dual $H$.
So it seems unlikely that a precompact group topology is the Mackey topology, this means also the finest compatible topology. However, in
  \cite{BTM} an example is given for a
    non-complete metrizable precompact group topology
   which is Mackey. This was taken up in  \cite{tesislorenzo} where it was shown that for every $m\ge 2$ the group $(G_m,\tau_m)=(\bigoplus_{n\in\N} \Z/m\Z,\tau_m)$ endowed with the
   topology $\tau_m$ induced by the product is the unique locally quasi-convex group topology with dual $(G_m,\tau_m)^\wedge$.
   In particular,  $\tau_m$ is the Mackey topology.
   The topology $\tau_m$ is {\bf linear}, which means that  it has a neighborhood basis at $0$ consisting of
(necessarily open) subgroups. It is trivial to see that every linear group topology is locally quasi-convex.

Since in the vector space case every locally convex metrizable vector space topology is the Mackey topology, the following questions are natural:
\begin{enumerate}
\item Let $(G,\tau)$  be a metrizable  locally quasi-convex group. Is $\tau$ the Mackey topology?
\item Let $(G,\lambda)$ be a metrizable group and $\lambda$ a linear group topology. Is $\lambda$ the Mackey topology?
\item {\bf (Dikranjan)} Are the $p$-adic topologies on $\Z$    Mackey topologies?
\end{enumerate}

 Of course, the questions above are of decreasing power.

   The first question was answered in the negative in
  the recent preprint \cite{28diciembre}. It is shown there that for a compact, connected, metrizable abelian group $X$ the group
  of null-sequences $ {\rm c}_0(X)$ endowed with the topology of uniform convergence $\tau$ is complete, metrizable, and locally quasi-convex, hence Mackey. It is shown in the paper that  the dual group $H:=({\rm c}_0(X),\tau)^\wedge$ is countable. Hence the weak topology $\sigma=\sigma({\rm c}_0(X),H)$,
  is also metrizable, locally quasi-convex,  and strictly coarser than $\tau$ and compatible with $ \tau$. This shows that $({\rm c}_0(X),\sigma)$ is metrizable, but
  not the  Mackey topology.

In this article we show that every non-discrete  Hausdorff  linear topology on $\Z$ is not Mackey, in
  particular the $p$-adic topologies are not Mackey (Theorem \ref{linnotM}).
This gives a negative  answer to (3) and shows that (2) is not always true. (However, the example of the groups $G_m$ above show, that there do exist
metrizable linear topologies which are Mackey.)

\vspace{0.3cm}

Now we will sketch the basic idea how to construct a group topology   on $\Z$ which is compatible with a given linear topology $\lambda$ and
 strictly finer than   $\lambda$.

Given a non-discrete  Hausdorff  linear group topology on the integers $\Z$, there exists a strictly increasing sequence $(b_n)_{n\in\N_0}\in \N^{\N_0}$ such that
$(b_n\Z)_{n\in\N_0}$ is a neighborhood basis at $0$ and $b_0=1$, and $b_n|b_{n+1}$ (Proposition \ref{topologiaslinealesnosondiscretas}). Conversely, given such a sequence $\textbf{b}=(b_n)_{n\in\N_0}$,
the family $(b_n\Z)_{n\in\N_0}$ is a neighborhood basis of a linear group topology on $\Z$, which we will denote by $\lambda_{\textbf{b}}$. It is easy to prove  that the dual group $H_{\textbf{b}}:=(\Z,\lambda_{\textbf{b}})^\wedge$ of $(\Z,\lambda_{\textbf{b}})$ can be identified with $\{\frac{k}{b_n}+\Z:\ k\in\Z,\ n\in\N_0\}$ (Proposition \ref{duallineal}). Since $\lambda_{\textbf{b}}$ is precompact, we have $\lambda_{\textbf{b}}=\sigma(\Z,H_{\textbf{b}})$.
  For a  prime $p$ and $(b_n)_{n\in\N_0}=(p^n)_{n\in\N_0}$    we obtain the so-called {\bf  $p$-adic topology} on $\Z$, the dual of which is the Pr\"ufer group
$\Z(p^\infty)$.

Suppose for the moment that there exists a compatible group topology $\tau_{\textbf{b}}$ for $\lambda_{\textbf{b}}$ which is strictly finer
than $\lambda_{\textbf{b}}$.
All topologies under consideration are locally quasi-convex. In order to describe the topology, it is sufficient to consider {\bf equicontinuous subsets} in the dual: these are subsets of polars of $0$-neighborhoods. Indeed, if $U$ is a quasi-convex neighborhood of $0$ then
$U=(U^\triangleright)^\triangleleft$ holds and $U^\triangleright$ is equicontinuous.
So let us first have a look on the dual homomorphism of the continuous identity mapping $\iota:(\Z,\tau_{\textbf{b}})\to
(\Z,\lambda_{\textbf{b}})$
 $$\iota^\wedge:
(\Z,\lambda_{\textbf{b}})^\wedge \longrightarrow  (\Z,\tau_{\textbf{b}})^\wedge.
 $$
The dual homomorphisms $\iota^\wedge$ (which is the identity map) is continuous. We endow the character groups with the compact-open topology.
It is a standard fact that every polar of a neighborhood of $0$ is a compact subset in the dual  group (e.g. (3.5) in \cite{tesislydia})
and it is well-known that
  the equicontinuous subsets in the dual of a precompact group, in particular in the dual of  $(\Z,\lambda_{\textbf{b}}) $, are
finite ((2.4) in \cite{28diciembre}).

 Hence, in order to obtain a  compatible topology $\tau_{\textbf{b}}$ which is strictly finer than $\lambda_{\textbf{b}}$,
the dual group $(\Z,\tau_{\textbf{b}})^\wedge$ must contain infinite equicontinuous, and hence compact, quasi-convex subsets.\\
Since $\phi:(\Z,\tau_{\textbf{b}})^\wedge\to (\T,\sigma(\T,\Z)) ,\ \chi\mapsto\chi(1)$ is continuous, for any neighborhood $U$ at $0$ in
$(\Z,\tau_{\textbf{b}})$ the sets $\phi(U^\triangleright)$ are also compact in the compact group
 $\T $ (observe, that $\sigma(\T,\Z)$ is the usual topology on $\T$). Since a homomorphism $(\Z,\tau_{\textbf{b}})\to\T$ which maps a neighborhood $U$ of $0$ in
 $(\Z,\tau_{\textbf{b}})$ in $\T_+$ is continuous, it is straight forward to show that $\phi(U^\triangleright)$ is also quasi-convex in $\T$.

In \cite{tesislorenzo}, \cite{DP}, \cite{DG}, and \cite{DG2} infinite quasi-convex null-sequences in (some) LCA groups have been established.
A typical candidate for a compatible topology could be the topology of uniform convergence on such a   null-sequence in $\T$.
(For more details, see the comment at the beginning of section 3.)

\section{Linear topologies on $\Z$}
In this section we recall some results about linear topologies on $\Z$. They give in a canonical way rise to
sequences, which we shall call $D$-sequences.

First we recall a characterization of linear topologies on $\Z$, since this will motivate the definition of these $D$-sequences.

\begin{proposition}\label{topologiaslinealesnosondiscretas}

The following statements are equivalent:

\begin{enumerate}

\item[(1)] $\lambda$ is a non-discrete Hausdorff linear group topology on $\Z$.

\item[(2)] There exists a sequence $\textbf{b}=(b_n)_{n\in\N_0}$ in $\Z$ with $b_0=1$, $b_n\neq b_{n+1}$ and $b_n\mid b_{n+1}$ for all $n\in\N_0$ such that $\{b_n\Z\mid n\in\N_0\}$ is a neighborhood basis at $0$ in  $(\Z,\lambda)$.

\end{enumerate}

\end{proposition}

\begin{proof}
\noindent  $(1)\Rightarrow (2)$
Let  $\mathcal{U}=(U_i)_{ i\in I} $ be  a neighborhood basis at $0$  consisting of subgroups.
Hence  $ \{U_i :\ {i\in I}\}=\{b_n\Z:\ {n\in\N_0}\}$ for suitable $b_n\in\N$.
Without loss of generality  we may
assume that $b_0=1$    and that $b_n\Z\supsetneq b_{n+1}\Z$ for all $n\in\N$, which is equivalent to $b_n|b_{n+1}$ and $b_n\not=b_{n+1}$.

\noindent$(2)\Rightarrow (1)$ follows easily from  proposition 1 , chapter III, $\S$ 1.2 in \cite{bourbaki}.
The induced topology is Hausdorff, since  $x\in b_n\Z$ for all $n\in\N$ iff $b_n$ divides $ x$ for all $n\in\N$ iff $x=0$.
\end{proof}

\begin{remark}
\begin{enumerate}
\item[(1)]
Under the conditions of the above proposition, the topology $\lambda$ is metrizable by the
  Birkhoff-Kakutani theorem (see e.g. 3.3.12 in \cite{teoremaBK}).
\item[(2)] Given a sequence $\textbf{b}=(b_n)_{n\in\N_0}$ as in Proposition  \ref{topologiaslinealesnosondiscretas} (2), we denote
the associated linear topology by $\lambda_{\textbf{b}}$.
\item[(3)] For $(b_n)_{n\in\N_0}=(p^n)_{n\in\N_0}$ the topology $\lambda_{\textbf{b}}$ is the {\bf  $p$-adic topology}.

\end{enumerate}
\end{remark}

 The representation in  Proposition \ref{topologiaslinealesnosondiscretas} (2) motivates the following definition.

\begin{definition}\label{definicionDsequence}
  Let \textbf{b}$=(b_n)_{n\in\N_0}$ be a sequence of natural numbers. We say that \textbf{b} is a {\bf  $D$-sequence} if $b_0=1$ and $b_n\mid b_{n+1}$ and $b_n\neq b_{n+1}$ for all $n\in\N_0$. We will write $\mathcal{D}:=\{\textbf{b}\mid \textbf{b}$ is a $D$-sequence$\}$ and $\mathcal{D}_\infty:=\{\textbf{b}\mid\textbf{b}$ is a $D$-sequence and $\frac{b_{n+1}}{b_n}\rightarrow \infty\}$.

\end{definition}

\begin{example}

\begin{enumerate}
\item[(1)]  $\left( a^n\right)_{n\in\N_0}\in \mathcal{D}\setminus  \mathcal{D}_\infty$ for every $a\in\N\setminus\{1\}$.
\item[(2)] $ \left( (n+1)!\right)_{n\in\N_0}\in \mathcal{D}_\infty$.
\item[(3)] $ \left( a^{n^2}\right)_{n\in\N_0}\in \mathcal{D}_\infty$ for every    natural number $a>1$.
\end{enumerate}

\end{example}

\begin{remark}\label{acotaciondsequence}\label{lema20110301}
\begin{enumerate}
\item[(1)]
  Obviously, every subsequence of a $D$-sequence \textbf{b} which contains $b_0$ is a $D$-sequence.
\item[(2)]
  Let \textbf{b} be a $D$-sequence. Inductively, since $\frac{b_{n+1}}{b_n}\geq 2$, we have  $b_{n+k}\geq 2^kb_n$
  for all $n,k\in\N_0.$
 \item[(3)] For a $D$-sequence $\textbf{b}$ we have $\dis\sum_{n\ge j}\frac{1}{b_n}\leq\frac{2}{b_{j }}\quad\quad \forall\;j\in\N_0.$
   \end{enumerate}
 \noindent[Only (3) needs to be shown.
  We have $\dis \sum_{n\ge j}\frac{1}{b_n}=\sum_{k=0}^\infty \frac{1}{b_{j+ k}}\stackrel{(2)}{\leq}\sum_{k=0}^\infty\frac{1}{b_{j }}\frac{1}{2^k}=
\frac{1}{b_{j }}\sum_{k=0}^\infty\frac{1}{2^k}=\frac{2}{b_{j }}$ for all $j\in\N_0$.]
 \end{remark}

Since every non-trivial subgroup of $\Z$ has finite index, we have

\begin{proposition}\label{precompact}
  Let \textbf{b} be a $D$-sequence. Then $(\Z,\lambda_\textbf{b})$ is precompact.
\end{proposition}

 Recall, that every precompact topology is the weak topology induced by its dual.

\begin{lemma} \label{igualdadlineales1}
  Let \textbf{b} be a $D$-sequence and let $\textbf{c}\in \mathcal{D}$  be a subsequence. Then $\lambda_\textbf{b}=\lambda_\textbf{c}$.
\end{lemma}

\begin{proof}
  Since $\textbf{c}$ is a subsequence of $ \textbf{b}$ we have that $\{c_n\Z\mid n\in\N_0\}\subseteq\{b_n\Z\mid n\in\N_0\}$ and hence $\lambda_\textbf{c}\leq\lambda_\textbf{b}$. Conversely, for $n\in\N$, there exist  $m\in\N$ and $n_m\in\N$ such that $c_m=b_{n_m}>b_n$ and hence   $b_n\mid b_{n_m}=c_m$, which implies $c_m\Z\subseteq b_n\Z$. This shows that $\lambda_\textbf{c}\ge\lambda_\textbf{b}$.
\end{proof}

\begin{proposition}\label{igualdadlineales}
For every $\textbf{b}\in\mathcal{D}$  there exists $\textbf{c}\in\mathcal{D}_\infty$ such that $\lambda_\textbf{b}=\lambda_\textbf{c}$.

\end{proposition}

\begin{proof}
 We shall define recursively  a suitable subsequence \textbf{c}.  Fix $c_0=b_0=1$ and $c_1=b_1$.  Given $c_0,\dots,c_n$ we define $c_{n+1}$ in the following way: For $i$ such that $c_n=b_i$ there exists $j\in\N_0$ such that $\frac{b_j}{b_i}>\frac{c_{n }}{c_{n-1 }}$; or equivalently, $\frac{b_j}{c_n}>\frac{c_{n}}{c_{n-1}}$. Define $c_{n+1}:=b_j$. This implies that $\frac{c_{n+1}}{c_{n }}>\frac{c_n}{c_{n-1}}$. Since $\frac{c_n}{c_{n-1}}\in\N$, we have $\frac{c_{n+1}}{c_{n}}\rightarrow \infty$, which is equivalent to $\frac{c_{n}}{c_{n+1}}\rightarrow 0.$ Hence $\textbf{c}\in\mathcal{D}_\infty$.  By Lemma \ref{igualdadlineales1}, we have $\lambda_\textbf{b}=\lambda_\textbf{c}$.
\end{proof}

For the sake of completeness, we recall here the structure of the dual of $(\Z,\lambda_{\textbf{b}})$. Therefore, we introduce the
following notation:

\begin{notation}
For every $\textbf{b}\in\mathcal{D}$ and every $n\in\N_0$, we define $$\xi^{\,\textbf{b}}_n:\Z\to \T,\ k\mapsto \frac{k}{b_n}+\Z .$$ If no confusion can arise,
we simply write $\xi_n$ instead of $\xi_n^{\,\textbf{b}}$.
Of course, $\langle\{\xi_m^{\,\textbf{b}}:\ m\in\N\}\rangle=\bigcup_{m\in\N}\langle\xi_m^{\,\textbf{b}}\rangle $ holds.
\end{notation}

\begin{proposition}\label{duallineal}
For $\textbf{b}=(b_n)\in\mathcal{D}$, $\quad $
 $(\Z,\lambda_\textbf{b})^\wedge=\langle \{\xi_n^{\,\textbf{b}}:\ n\in\N\}\rangle$ holds.
 In particular, $\lambda_{\textbf{b}}$ is the weak topology, that is the coarsest topology which makes all characters $\chi\in (\Z,\lambda_{\textbf{b}})^\wedge$ continuous.
 Furthermore, every equicontinuous
 subset of $(\Z,\lambda_\textbf{b})^\wedge$ is finite.
\end{proposition}

\begin{proof}
Fix $n\in\N$ and $\chi\in (b_n\Z)^\triangleright$. Since $\chi(b_n\Z)$ is a subgroup contained in $\T_+$, this  yields $\chi(b_n\Z)=\{0+\Z\}$. For $x\in\R$ such that $\chi(1)=x+\Z$, this implies $ b_nx+\Z=\chi(b_n)=0+\Z$ and
hence
 there exists $k\in\Z$ such that $x=\frac{k}{b_n}$. So $(b_n\Z)^\triangleright\subseteq \langle\xi_n^{\,\textbf{b}}\rangle$.
 In particular, this set is finite.

Now we show that $(\Z,\lambda_\textbf{b})^\wedge\subseteq\langle \{\xi_n^{\,\textbf{b}}:\ n\in\N\}\rangle$ holds.
Fix $\chi\in(\Z,\lambda_\textbf{b})^\wedge$. Since $\chi$ is continuous,
there exists a neighborhood $b_n\Z$ such that $\chi(b_n\Z)\subseteq\T_+$.
As shown above, this implies that $\chi \in  \langle\xi_n^{\,\textbf{b}}\rangle\subseteq \langle \{\xi_m^{\,\textbf{b}}:\ m\in\N\}\rangle$.

 Conversely, fix $ n\in\N_0$ and $k\in\Z$. The  kernel of the homomorphism $\chi=k\xi_n^{\,\textbf{b}}:\Z\rightarrow\T$, $l\mapsto \frac{kl}{b_n}+\Z$ contains the neighborhood $b_n\Z$ and therefore, $\chi$ is continuous. Since $k$ and  $n$ were arbitrary, the assertion follows.
 By Proposition \ref{precompact}, $(\Z,\lambda_{\textbf{b}})$ is precompact. This implies that $\lambda_{\textbf{b}}$ is the weak topology induced by
 $\langle \{\xi_n^{\,\textbf{b}}:\ n\in\N\}\rangle$.
\end{proof}

\begin{proposition}\label{convlambda}\label{bnconvergeenlandab}
Let $(b_n)_{n\in\N_0}$ be a $D$-sequence and consider  a sequences of integers
   $(l_j)_{j\in\N}$. Then the following conditions are equivalent:

\begin{enumerate}
\item[(1)] $l_j\rightarrow 0$ in $\lambda_\textbf{b}$.
\item[(2)] For every $n\in\N$,\ $\xi_n^{\,\textbf{b}}(l_j)\rightarrow 0+\Z$.
\item[(3)]   For every $n\in\N$ there exists $j_n$ such that $b_n\mid l_j$ for all $j\geq j_n$.
\end{enumerate}
In particular,  $(b_n)$ converges to $0$ in $\lambda_\textbf{b}$.
\end{proposition}

\begin{proof}
(1) $\Longleftrightarrow$ (2) holds because of Proposition \ref{duallineal}.

(1) $\Longleftrightarrow$ (3):
  $(l_j)$ converges to $0$ in $\lambda_\textbf{b}$ iff for every $n\in\N$ there exists $j_n$ such that $l_j\in b_n\Z$ for all $j\geq j_n$. This condition is equivalent to $b_n\mid l_j$ for all $j\geq j_n$.

  The additional assertion follows immediately, since $\textbf{b}$ is a $\mathcal{D}$-sequence.
\end{proof}

 \begin{definition}[Protasov, Zenlenyuk; Barbieri, Dikranjan, Milan, Weber;\cite{protasov}, \cite{BDMW}]

 Let G be a an abelian group. We say that a sequence \textbf{u}$=(u_n)\in G^{\N}$ is a \textbf{$T$-sequence} if there exists a Hausdorff group topology on $G$ in which $(u_n)$ converges to zero. We  denote by $\mathcal{T}_G$ the set of all $T$-sequences in $G$.
 If there exists a precompact Hausdorff group topology on $G$ in which $(u_n)$ converges to $0$, the sequence is
 called a \textbf{$TB$-sequence}. The set of all {$TB$-sequences} in $G$  is denoted by $\mathcal{TB}_G$.
 \end{definition}

As a consequence of Propositions  \ref{precompact} and \ref{bnconvergeenlandab}, we obtain
 $\mathcal{D}_\infty\subseteq\mathcal{D}\subseteq\mathcal{TB}_\Z\subseteq \mathcal{T }_\Z.$

\section{Constructing topologies of uniform convergence associated with a $D$-sequence}
The aim of this section is to construct a locally quasi-convex group topology which is strictly
finer than $\lambda_{\textbf{b}}$ where $\textbf{b}$ is a $D$-sequence.

As explained in the introduction, it is necessary to find a compact and quasi-convex subset of  $\T$ contained in $H_{\textbf{b}}=(\Z,\lambda_{\textbf{b}})^\wedge$.

 A first attempt was made in \cite{tesislorenzo}, where several quasi-convex null-sequences in the Pr\"ufer-group
$\Z(2^\infty)\subseteq\T$ have been  presented. This was taken up in \cite{DP}, \cite{DG}, \cite{DG2} where quasi-convex sequences
in locally compact abelian groups have been studied.

Let $ S_{\textbf{b}}= \{\frac{1}{b_n}+\Z:\ n\in\N\} $ where $\textbf{b}\in \mathcal{D}_\infty$.
In \cite{tesislorenzo}, the topology of uniform convergence   on a quasi-convex null sequence contained in $ \Z(2^\infty)$  was introduced and basic properties have been  established.
Here we consider the topology $\tau_{\textbf{b}}$ of uniform convergence
on $ S_{\textbf{b}} $. The main theorem of this article (Theorem \ref{dualsicocientesacero}) asserts that if
 $\textbf{b}\in \mathcal{D}_\infty$, then $ (\Z,\tau_{\textbf{b}})^\wedge =\langle \{\frac{1}{b_n}+\Z:\ n\in\N\}\rangle=(\Z,\lambda_{\textbf{b}})^\wedge$, where $\langle S\rangle$ denotes the
subgroup generated by $S$. As a consequence, we obtain that every  non-discrete  Hausdorff  linear topology on $\Z$ is not Mackey.

 In \cite{DG} conditions for the set $ \pm S_{\textbf{b}}\cup\{0+\Z\}$ are given such that it is quasi-convex.
However, in order that $\tau_{\textbf{b}}$ be compatible with the weak topology $ \sigma(\Z,(\Z,\lambda_{\textbf{b}})^\wedge)$,
not only $ \pm S_{\textbf{b}}\cup\{0+\Z\}$, but also $   (S_{\textbf{b}}\stackrel{{\tiny\mbox{n summands}}}{+\ldots+}S_{\textbf{b}})^{\triangleleft\triangleright}$ must be  compact subsets of $ \T$  contained in $(\Z,\lambda_{\textbf{b}})^\wedge$ for all $n\in\N$.

Indeed, let $S_{\textbf{b}}\subseteq U^\triangleright$ for a suitable $0$-neighborhood $U$. For every $n\in\N$ there exists a quasi-convex neighborhood $W$ of $0$ such that $W\stackrel{{\tiny\mbox{n summands}}}{+\ldots+}W\subseteq U$ and hence $S_{\textbf{b}}\subseteq U^\triangleright\subseteq (W\stackrel{{\tiny\mbox{n summands}}}{+\ldots+}W)^\triangleright.$ It is straightforward to check that for $x\in W$ and $\chi\in S_{\textbf{b}}$ one has $\chi(x)\in[ -\frac{1}{4n},\frac{1}{4n}]+\Z$ and hence
 $S_{\textbf{b}}\stackrel{{\tiny\mbox{n summands}}}{+\ldots+}S_{\textbf{b}}\subseteq
W^\triangleright$. Since $W^\triangleright$ is quasi-convex, we obtain that the quasi-convex hull ${\rm qc}(S\stackrel{{\tiny\mbox{n summands}}}{+\ldots+}S)$ of $S_{\textbf{b}}\stackrel{{\tiny\mbox{n summands}}}{+\ldots+}S_{\textbf{b}}$ (the smallest quasi-convex set
containing $S_{\textbf{b}}$)   is contained in $W^\triangleright$.
 However, up to now, there is no algorithm  known   to  construct such compact sets.

We start with recalling some properties of the topology of uniform convergence.

\begin{proposition}\label{Stopologiassontopologias}
 Let $G$ be a topological group, and $\emptyset\not=S\subseteq \Hom (G,\T)$. We denote by $\T_n=[-\frac{1}{4n},\frac{1}{4n}]+\Z$.
 (Note that $\T_1=\T_+$ holds.)
 The family
 $\mathcal{U}_S= \displaystyle(\bigcap_{\chi\in S}\chi^{-1}(\T_{n})\mid n\in\N)_{n\in\N}$ forms a neighborhood basis at $0$ for a locally quasi-convex group topology, which we denote by   $\tau_S$. It is called the \textbf{topology of uniform convergence on $S$}. This topology is Hausdorff iff $S$ separates the points of $G$ (i.e. for every $x\in G\setminus \{0\}$ there exists $\chi\in S$ with $\chi(x)\not=0+\Z$).
\end{proposition}

\begin{proof}
 Denote by $U_{n}:=\displaystyle\bigcap_{\chi\in S}\chi^{-1}(\T_{n})$. The family $(U_n)$ is a decreasing sequence
 of symmetric sets containing $0$. The inclusion $\T_{2n}+\T_{2n}\subseteq \T_n$ yields that $U_{2n}+U_{2n}\subseteq U_n$. This shows that
  $(U_n)$ forms a neighborhood basis at $0$ of a group topology (proposition 1 , chapter III, $\S$ 1.2 in\cite{bourbaki}).
  Furthermore,   $\tau_S$ is a Hausdorff topology iff $\bigcap_{n\in\N}U_n=\bigcap_{\chi\in S}\chi^{-1}(\{0\})=\{0\}$, or equivalently, if $S$ separates the points of $G$.

 Finally, it remains to be shown that the sets $U_n$ are quasi-convex. This is an easy consequence of
 the following facts:
 Intersections and inverse images under continuous homomorphisms of quasi-convex sets are quasi-convex (e.g. (6.2) in \cite{tesislydia}). It is straight forward to check that the sets $\T_m$ are quasi-convex. Hence the assertion follows.
\end{proof}

Now we consider a particular family of topologies of uniform convergence for the group  $\Z$.

\begin{notation}
  Given a D-sequence \textbf{b}$=(b_n)$, we denote the $S$-topology corresponding to   the set $\{\xi_n^{\,\textbf{b}}:\ n\in\N\}\subseteq\Z^\wedge$ (where
  $\xi_n^{\,\textbf{b}}:\Z\to \T,\ k\mapsto \frac{k}{b_n}+\Z$)  by   $\tau_{\textbf{b}}$. Further, we define
$$V_{\textbf{b},m}=\{k\in\Z: \  \frac{k}{b_n}+\Z\in\T_m\ \mbox{for all}\ n\in\N\}=\{k\in\Z: \  \xi_n^{\,\textbf{b}}(k)\in\T_m\ \mbox{for all}\ n\in\N\}.$$ Then, according to  Proposition  \ref{Stopologiassontopologias},  $(V_{\textbf{b},m})_{m\in\N}$ is  a neighborhood basis at $0$ for $(\Z,\tau_{\textbf{b}})$ consisting of
quasi-convex sets.
Note   that $V_{\textbf{b},1}=\{\xi_n^{\,\textbf{b}}:\ n\in\N\}^\triangleleft$.
\end{notation}

 \begin{remark}
 The topology $\tau_{\textbf{b}}$ is  Hausdorff: If $k$ is an integer belonging to  $ \bigcap_{m\in\N}V_{\textbf{b},m}$,  then   $\frac{k}{b_m}+ \Z=0+\Z$ for all $m\in\N$, hence
$b_m$ divides $k$ for all $m\in\N$ and so $k=0$.

By the Birkhoff-Kakutani theorem, $\tau_{\textbf{b}}$ is metrizable.

It is straight forward to see that the mapping $(\Z,\tau_{\textbf{b}})\to {\rm c}_0(\T),\ k\mapsto (\frac{k}{b_n}+\Z)$ is an embedding.
\end{remark}

\begin{remark} \label{tauSesdiscreta}
  Let \textbf{b} be a $D$-sequence.
  It can be shown that  $\left(\frac{b_{n+1}}{b_n}\right)$ is bounded if  $\tau_\textbf{b}$ is discrete.

  Indeed, suppose that $\frac{b_n}{b_{n+1}}>\frac{1}{ M}$ for a suitable natural number $M$.
For every $k\in\Z\setminus\{0\}$ there is a unique $n\in\N_0$ such that
  $$\begin{array}{ccccccc}
 & &\frac{b_n}{4}&\le &|k|&<&\frac{b_{n+1}}{4}\\[1ex]
  \Longrightarrow&\frac{1}{4M}<&\frac{b_n}{4b_{n+1}}&\le &\frac{|k|}{b_{n+1}}&<&\frac{1}{4}.
  \end{array}$$
This shows that $k\notin V_{\textbf{b},m}$ or equivalently that $V_{\textbf{b},m}=\{0\}$.
\end{remark}

\begin{proposition}\label{convergenciaentau_S}
  Let $(l_j)_{j\in\N}$ be a sequence of integers. Then the following conditions are equivalent:
\begin{enumerate}

\item[(1)] $l_j\rightarrow 0$ in $\tau_\textbf{b}$.

\item[(2)] $(\xi_n^{\,\textbf{b}}(l_j))_{j\in\N}$ converges uniformly in $n$ to $0+\Z\in\T$.

\item[(3)] For every $m\in\N$, there exists $j_m\in\N$ such that for all $n\in\N$ and all $j\geq j_m\quad \underbrace{\frac{l_j}{b_n}+\Z}_{=\xi_n^{\,\textbf{b}}(l_j)}\in\T_m$ holds.
\end{enumerate}

\end{proposition}

\begin{proof}
The sequence
 $(l_j)$ converges to $0$ in $\tau_{\textbf{b}}$ iff for every $m\in\N$, there exists $j_m$ such that for all $j\geq j_m$ we have $l_j\in V_{\textbf{b},m}$. The last condition is equivalent to $\xi_n^{\,\textbf{b}}(l_j)\in \T_m$ for all $n\in\N$ and all $j\ge j_m$. This shows the equivalence between $(1)$ and $(2)$. Since $\xi_n^{\,\textbf{b}}(l_j)=
  \frac{l_j}{b_n}+\Z $,  also the equivalence with $(3)$ is clear.
\end{proof}

\begin{corollary}\label{bconvergeaceroentauS}
  For  $\textbf{b}\in \mathcal{D} $, the following statements are equivalent:
  \begin{enumerate}
  \item[(1)] $\textbf{b}\in \mathcal{D}_\infty .$
  \item[(2)]  $b_j\stackrel{\tau_\textbf{b}}{\rightarrow} 0$.
  \end{enumerate}
\end{corollary}

\begin{proof}

\noindent (1) $\Longrightarrow$ (2): Fix $m\in \N$. Since $\frac{b_{j}}{b_{j+1}}\rightarrow 0$ in $\R$, there exists $j_m$ such that $\frac{b_j}{b_{j+1}}\le\frac{1}{4m}$ for all $j\geq j_m$. Choose $j\geq j_m$.

\noindent  If $n\leq j\quad \quad \frac{b_j}{b_n}+\Z=0+\Z\in\T_m$.

\noindent If $n>j\quad \quad \left|\frac{b_j}{b_n}\right|\leq\frac{b_j}{b_{j+1}}\le\frac{1}{4m}$; which implies that $\frac{b_j}{b_n}+\Z\in\T_m$  for all $n\in\N$ and $j\geq j_m$.\\
Combining these facts, we conclude from Proposition \ref{convergenciaentau_S} that
 $(b_j)$ converges to $0$ in ${\tau_\textbf{b}} $.

 \noindent (2) $\Longrightarrow$ (1):
 Assume that $(b_j)$ converges to $0$ in  ${\tau_\textbf{b}} $. This implies that for given $m\in\N$ there exists $j_m\in\N$ such that
 $\frac{b_j}{b_n}+\Z\in\T_m$ holds for all $j\ge j_m$ and all $n\in\N$. In particular $\frac{b_j}{b_{j+1}}+\Z\in\T_m$ for all $j\ge j_m$.
 This is equivalent to $\frac{b_j}{b_{j+1}}\le \frac{1}{4m}$, so (1) follows.

\end{proof}

\begin{proposition}\label{remark3.3}
 For every $D$-sequence  \textbf{b}, the topology $\tau_{\textbf{b}}$ is strictly finer than $\lambda_{\textbf{b}}$. In particular, $\id:(\Z,\tau_{\textbf{b}})\to (\Z,\lambda_{\textbf{b}})$ is a continuous isomorphism.
\end{proposition}

\begin{proof}
  Since both topologies are metrizable, it is sufficient to consider sequences. So let $(l_j)_{j\in\N}$ be a sequence which converges to $0$ in $\tau_\textbf{b}$. According to Proposition \ref{convergenciaentau_S}, $\xi_n^{\,\textbf{b}}(l_j)$ converges uniformly in $n$ to $0+\Z$, in particular it converges pointwise. So it is a consequence of Proposition \ref{convlambda} that $(l_j)$ converges to $0$ in $\lambda_{\textbf{b}}$.

 In order to prove that $\tau_{\textbf{b}}$ is strictly finer than $\lambda_{\textbf{b}}$, we consider
 equicontinuous subsets in the dual. According to Proposition \ref{duallineal}, all equicontinuous subsets of $(\Z,\lambda_{\textbf{b}})^\wedge$
 are finite. However, it is straight-forward to check that $\{\frac{1}{b_n}+\Z:\ n\in\N\}\subseteq V_{\textbf{b},1}^\triangleright
\subseteq V_{\textbf{b},m}^\triangleright$ for all $m\in \N$. Since $\{\frac{1}{b_n}+\Z:\ n\in\N\}$ is   infinite, these two topologies cannot coincide and hence  $\tau_{\textbf{b}}$ is strictly finer than $\lambda_{\textbf{b}}$.
\end{proof}

Since $\id:(\Z,\tau_{\textbf{b}})\to (\Z,\lambda_{\textbf{b}})$ is continuous, every $\lambda_{\textbf{b}}$-continuous
character is  $\tau_{\textbf{b}}$-continuous, hence we obtain:

\begin{corollary}\label{remark3.4}
 $(\Z ,\lambda_{\textbf{b}})^\wedge $ is a subgroup of $ (\Z,\tau_{\textbf{b}})^\wedge$.
\end{corollary}


\section{The dual group of $(\Z,\tau_\textbf{b})$ for  $\textbf{b}\in\mathcal{D}_\infty$}

The aim of this section is to determine the dual of $(\Z,\tau_\textbf{b})$ for all sequences
$\textbf{b}\in\mathcal{D}_\infty$.

\begin{proposition}\label{5.4.4}

   If integers $k_0,\ldots, k_N$ satisfy $\dis\  \left| k_j \right|\leq\frac{1}{8m}\cdot\frac{b_{j+1}}{b_j}$   for all $0\le j\le N$, then $\dis k=\sum_{j=0}^Nk_jb_j\in V_{\textbf{b},m}$.

\end{proposition}

\begin{proof}
  By assumption, we have $\dis \frac{\mid k_j\mid b_j}{b_{j+1}}\leq \frac{1}{8m}$ for all $j\in\N_0$.
Hence, for $1\le n\le N$  we obtain $$\dis   \left|\sum_{j=0}^{n-1}\frac{k_jb_j}{b_n}\right|\le \sum_{j=0}^{n-1}\frac{b_{j+1}}{b_n}\frac{\mid k_j\mid b_j}{b_{j+1}}\leq \frac{1}{8m}\sum_{j=0}^{n-1}\frac{b_{j+1}}{b_n} \stackrel{\ref{acotaciondsequence}}{\leq}\frac{1}{8m} \sum_{i=0}^{n-1}\frac{1}{2^i}\leq\frac{1}{4m}.$$
This yields $\dis \xi_n^{\textbf{b}}(k)=\frac{k}{b_n}+\Z=\frac{1}{b_n}\sum_{j=0}^Nk_jb_j+\Z=\frac{1}{b_n}\sum_{j=0}^{N\wedge n-1} k_jb_j+\Z\in\T_m$
for all $n\in\N$. So the assertion follows from Proposition \ref{convergenciaentau_S}.
\end{proof}

In \cite{unpublished} a proof of the following lemma can be found:

\begin{lemma}\label{1/bn}
For any $x\in\left(-\frac{1}{2},\frac{1}{2}\right]$, there exists a representation $\dis x=\sum_{n\in\N}\frac{d_n}{b_n}$, where $\dis \left| {d_n} \right|\leq\frac{b_n}{2b_{n-1}} $ and $d_n\in\Z$ for all $n\in\N$.
\end{lemma}

 \begin{lemma}\label{wwi}Let $\textbf{b}=(b_n)_{n\in\N_0}$ be a $D$-sequence and let $\dis x=\sum_{n\in\N}\frac{d_n}{b_n}$
 be as in Lemma \ref{1/bn}.
For
 $$b_jx+\Z =\sum_{n\ge j+1}b_j\frac{d_n}{b_n}+\Z=\underbrace{b_j\frac{d_{j+1}}{b_{j+1}}+b_j\frac{d_{j+2}}{b_{j+2}}}_{=:e_j}+\underbrace{b_j\sum_{n\ge j+3}\frac{d_n}{b_n}}_{=:\eps_j}+\Z$$
 the following estimates hold for all $j\in\N_0$:
 \begin{enumerate}
  \item $\dis |e_j|\le \frac{3}{4}$,
  \item if $d_{j+1}\not=0$, then $|e_j|\ge\frac{b_j}{2b_{j+1}}$, and
 \item $\dis |\eps_j|\le  \frac{b_j}{b_{j+2}}$.
 \end{enumerate}
  \end{lemma}

\begin{proof}

 \noindent(1) We have $\dis |e_j| \le  b_j\frac{|d_{j+1}|}{b_{j+1}}+b_j\frac{|d_{j+2}|}{b_{j+2}}\stackrel{\ref{1/bn}}{\le}
 \frac{1}{2}+\frac{b_j}{2b_{j+1}}\le\frac{3}{4}$.

  \noindent(2) Let $\dis d_{j+1}\not=0$, then $|e_j|\ge b_j\left(\frac{|d_{j+1 }|}{b_{j+1}}- \frac{|d_{j+2 }|}{b_{j+2}}\right)
 \stackrel{\ref{1/bn}}{\ge} b_j\left(\frac{|d_{j+1 }|}{b_{j+1}}- \frac{1}{2b_{j+1}}\right)\ge \frac{b_j}{2b_{j+1}}.$

 \noindent(3) We have $\dis |\eps_j|\le b_j \sum_{n\ge j+3}\frac{|d_n|}{b_n}\stackrel{\ref{1/bn}}{\le} b_j \sum_{n\ge j+3}\frac{1}{2b_{n-1}}=
 \frac{b_j}{2}\sum_{n\ge j+2}\frac{1}{b_n}\stackrel{\ref{lema20110301} (3)}{\le}\frac{b_j}{2}\cdot\frac{2}{b_{j+2}}=\frac{b_j}{b_{j+2}}$.
\end{proof}

\begin{theorem}\label{dualsicocientesacero}
  For  $\textbf{b}\in\mathcal{D}_\infty$ we have $(\Z,\tau_\textbf{b})^\wedge=\langle\{\frac{1}{b_n}+\Z\mid n\in\N\}\rangle=(\Z,\lambda_\textbf{b})^\wedge$.

\end{theorem}

\begin{proof}
It follows from  Corollary \ref{remark3.4} and  Proposition \ref{duallineal} that
$(\Z,\tau_\textbf{b})^\wedge\ge(\Z,\lambda_\textbf{b})^\wedge= \langle\{\frac{1}{b_n}+\Z\mid n\in\N\}\rangle $ holds.

  In order to prove the other inclusion,  we fix $\chi\in (\Z,\tau_\textbf{b})^\wedge$. Let  $  \left(-\frac{1}{2},\frac{1}{2}\right]\ni  x=\sum_{n=1}^\infty\frac{d_n}{b_n}$  (as in  Lemma \ref{1/bn}) satisfy $\chi(1)=x+\Z$.

\noindent \underline{Claim:} The sequence $\dis (e_j)_{j\ge_0}:=\left(b_j\frac{d_{j+1}}{b_{j+1}}+b_j\frac{d_{j+2}}{b_{j+2}}\right)_{j\ge 0}$ converges to  $  0$ in $\R$.

\noindent \underline{Proof of the claim:} Since by Lemma \ref{wwi} we have $\dis |e_j| \le \frac{3}{4}$,
it is sufficient to show that $e_j+\Z\to 0+\Z$ in $\T$.

  By Lemma \ref{bconvergeaceroentauS}, $(b_{j })$ tends to $0$ in $\tau_\textbf{b}$. Hence the continuity of $\chi$
  implies that $\chi(b_{j })=b_{j }x+\Z=e_j+\eps_j+\Z\longrightarrow 0+\Z$ for $j\to \infty$ in $\T$.

Since by Lemma \ref{wwi} $|\eps_j|\le\frac{b_j}{b_{j+2}}=\frac{b_j}{b_{j+1}}\frac{b_{j+1}}{b_{j+2}}$, the assumption $\textbf{b}\in{\mathcal D}_\infty$ implies that the sequence
$(\eps_j)$ converges to $0$ in $\R$. Together with $\lim_j (e_j+\eps_j) +\Z= 0+\Z$ this yields that $(e_j+\Z)$ tends to $0+\Z$ in $\T$ and proves the claim.

\vspace{0.3cm}

  Suppose that $A:=\{j \in\N:\ d_{j+1 }\neq 0\}$ is infinite.
  Since $\chi\in (\Z,\tau_\textbf{b})^\wedge$, there exists $m\in \N$ such
  that $\chi\in V_{\textbf{b},m}^\triangleright$. In order to obtain a contradiction,
  we want to find an integer $k\in V_{\textbf{b},m}$ such that $\chi(k)=xk+\Z\notin\T_+$.

By the claim and the hypothesis $\textbf{b}\in\mathcal{D}_\infty$   there exists  $j_0$ such that for all $j\ge j_0$

 $$(a)\
 \left| e_j\right| <  \frac{1}{96 m} \quad\quad\quad\mbox{and}\quad\quad\quad
  (b)  \   \left|\frac{b_j}{b_{j+1}}\right|<\frac{1}{6}$$ hold.

For $j\in A$ with $j\ge j_0$ we define $\dis k_{
j}:=\left\lfloor \frac{1}{16m|e_{j}|} \right\rfloor {\rm sign}(e_{
{j}})$. Observe that by Lemma \ref{wwi} (2), $e_j\not=0$ for all $j\in A$.

We have $\dis k_j e_j=\left\lfloor \frac{1}{16m|e_{j}|}
\right\rfloor|e_j|$, which implies
 $k_j e_j\le\frac{1}{16m}.$

 On the other hand, $\dis k_je_j\ge
 \left(\frac{1}{16m|e_{j}|}-1\right)|e_j|=\frac{1}{16m}-|e_j|\stackrel{(a)}{>}  \frac{1}{16m}-\frac{1}{96m}=\frac{5}{96m}.$

Further,  we have
$$|k_j\eps_j|\stackrel{\ref{wwi} (3)}{\le} \frac{1}{ 16m|e_{j}|} \cdot \frac{b_j}{b_{j+2}}
\stackrel{\ref{wwi} (2)}{\le} \frac{1}{16m}\cdot\frac{2b_{j+1}}{  b_j}\cdot\frac{b_j}{b_{j+2}}=
\frac{1}{ 8m}\frac{b_{j+1}}{b_{j+2}}\stackrel{(b)}{<}
\frac{1}{ 8m}\cdot \frac{1}{6}=\frac{1}{48m}.$$

Combining these estimates, we obtain

\begin{equation*} \frac{1}{32m}=\frac{5 }{96 m }-\frac{1}{48m} < k_j(e_j+\eps_j)< \frac{1 }{16m} +\frac{1}{48m}=\frac{1}{12m}.\quad\quad(\ast)
\end{equation*}

 We choose a subset $J$ of $A\cap \{j:\ j\ge j_0\}$ containing
 exactly $8m$ elements and define $k:=\sum_{j\in J}k_jb_j$.
 Since $$\dis \left|\frac{k_jb_j}{b_{j+1}}\right|\le \frac{1}{16m|e_j|}\cdot \frac{ b_j}{b_{j+1}}
 \stackrel{\ref{wwi}(2)}{\le}\frac{1}{16m}\cdot \frac{2b_{j+1}}{  b_j}\cdot\frac{ b_j}{b_{j+1}}=\frac{1}{8m}  ,$$ we obtain that $k\in
 V_{\textbf{b},m}$ (Proposition \ref{5.4.4}) and $\chi(k)=\sum_{j\in J}k_jx+\Z=\sum_{j\in
 J}k_j(e_j+\eps_j)+\Z$.
From ($\ast$) we conclude
 $$\frac{1}{4 }<  \sum_{j\in
 J}k_j(e_j+\eps_j)<\frac{2}{3}<
 \frac{3}{4 },$$
which yields  $\chi(k)=\sum_{j\in J}k_jx+\Z=\sum_{j\in J}k_j(e_j+\eps_j)+\Z\notin \T_+$. This is the desired contradiction.

\end{proof}

\begin{corollary}
The image under the inclusion $(\Z,\tau_{\textbf{b}})\to {\rm c}_0(\T),\ k\mapsto(\frac{k}{b_n}+\Z)_{n\in\N}$ is dually embedded.
\end{corollary}

\begin{proof}
It is sufficient to know that $({\rm c}_0(\T))^\wedge=\{\prod_{j=1}^Nk_j\pi_j :\ k_j\in\Z,\ N\in\N\}$
where $\pi_j:{\rm c}_0(\T)\to \T$ denotes the projection on the $j$-th  coordinate and $k_j\pi_j((x_n+\Z))=k_jx_j$.
\end{proof}

\begin{theorem}\label{linnotM}
 Let $\lambda$ be a non-discrete Hausdorff linear topology on $\Z$, then $\lambda$ is not Mackey. In particular the $p$-adic topologies are not Mackey.

\end{theorem}

\begin{proof}
 Since $\lambda$ is not discrete, $\lambda=\lambda_\textbf{b}$ for some $\textbf{b}\in\mathcal{D}$ (Proposition \ref{topologiaslinealesnosondiscretas}). According to Proposition \ref{igualdadlineales}, there exists $\textbf{c}\in\mathcal{D}_\infty$ such that $\lambda_\textbf{b} = \lambda_\textbf{c}$. By Theorem \ref{dualsicocientesacero}, we have $(\Z,\tau_{\textbf{c}})^\wedge=(\Z,\lambda_{\textbf{c}})^\wedge=
(\Z,\lambda_{\textbf{b}})^\wedge$. Since $\lambda_\textbf{b}=\lambda_\textbf{c}\stackrel{\ref{remark3.3}}{<}\tau_\textbf{c}$, the
linear topology $\lambda_{\textbf{b}}$ is not the Mackey topology.
\end{proof}

\section{Open Questions}

\begin{question}
Let $\textbf{b}\in\mathcal{D}$. Does there exist a Mackey topology for the pair $(\Z,\langle\{\frac{1}{b_n}+\Z\mid n\in\N\}\rangle)$?
\end{question}

In case it exists, it must be finer than all $\tau_c$ where $\textbf{c}\in\mathcal{D}_\infty$ and $\lambda_{\textbf{c}}=\lambda_{\textbf{b}}$.

\begin{question}
Fix $\textbf{b}\in\mathcal{D}$.
Is $s_{\textbf{b}}:=\sup\{\tau_{\textbf{c}}:\ \textbf{c}\in\mathcal{D}_\infty,\ \lambda_{\textbf{c}}=\lambda_{\textbf{b}}\}$ a compatible group topology?
\end{question}
 In case this is true, the following question arises:
\begin{question}
Is $s_{\textbf{b}}$ the Mackey topology?
\end{question}

The following four questions of decreasing power have been suggested by Dikran Dikranjan:

\begin{question}
Is there any metrizable locally quasi-convex group topology on $\Z$ which is Mackey?
\end{question}

\begin{question}
Is there any metrizable precompact group topology on $\Z$ which is Mackey?
\end{question}

\begin{question}
Is there any metrizable precompact  Mackey group topology on $\Z$   which does not have  proper open subgroups?
\end{question}

\begin{question}
Let $\alpha$ be an irrational number and denote by $\lambda_\alpha$ the initial topology induced on $\Z$ by the mapping
$\Z\to \T,\ k\mapsto \alpha k+\Z$.  Is  $\lambda_\alpha$ Mackey?
\end{question}

\noindent{\bf Acknowledgement}:
We would like to thank Elena Mart\'{\i}n Peinador and Vaja Tarieladze for useful comments and we are indebted to Dikran Dikranjan whose
 suggestions led to an improvement of the paper.
We also wish to thank him for his interest in this topic and the formulation of the questions 5.4 to 5.7.
Finally, we are indebted to the referee for his/her careful reading of the text, many useful suggestions and for
having pointed out an error in the former version.

\bibliographystyle{plain}

\end{document}